\documentclass{amsart}
\usepackage{enumerate}
\usepackage{color}
\usepackage[mathcal]{euscript}
\usepackage{mathrsfs}
\usepackage[T1]{fontenc}
\usepackage{graphicx}  
\usepackage[colorlinks=true]{hyperref}
\numberwithin{equation}{section}
\newtheorem{thm}{Theorem}[section]
\newtheorem{lem}[thm]{Lemma}
\newtheorem{rem}[thm]{Remark}

\theoremstyle{definition}

\numberwithin{equation}{section}

\DeclareMathOperator{\dist}{dist}

\newcommand{\A}{\mathcal{A}}

\newcommand{\C}{\mathcal{C}}

\newcommand{\N}{\mathbb{N}}
\newcommand{\Q}{\mathbb{Q}}

\newcommand{\R}{\mathbb{R}}

\begin{document}


\markboth{T. Drwi\k{e}ga and P. Oprocha}{$\omega$-chaos without infinite LY-scrambled set on Gehman dendrite}

\title{$\omega$-chaos without infinite LY-scrambled set on Gehman dendrite}

\author{TOMASZ DRWI\k{E}GA}

\address{Faculty of Applied Mathematics, AGH University of Science and Technology, \\
al. A. Mickiewicza 30, 30-059 Krak\'ow, Poland\\
drwiega@agh.edu.pl}

\author{PIOTR OPROCHA}
\address{Faculty of Applied Mathematics, AGH University of Science and Technology, \\
al. A. Mickiewicza 30, 30-059 Krak\'ow, Poland\\
	-- and --\\
	National Supercomputing Centre IT4Innovations, Division of the University of Ostrava, \\
	Institute for Research and Applications of Fuzzy Modeling, \\
	30. dubna 22, 70103 Ostrava,
	Czech Republic\\
	oprocha@agh.edu.pl}

\begin{abstract}
We answer the last question left open in [Z.~Ko\v{c}an, \emph{Chaos on one-dimensional compact metric spaces}, Internat. J. Bifur. Chaos Appl. Sci. Engrg. \textbf{22}, article id: 1250259 (2012)] which asks whether there is a relation between an infinite LY-scrambled set and $\omega$-chaos for dendrite maps. We construct a continuous self-map of a dendrite without an infinite LY-scrambled set but containing an uncountable $\omega$-scrambled set. 
\end{abstract}

\maketitle


\section{Introduction}

In 2012 in \cite{Koc12} characterized various chaotic properties of maps on dendrites. In this paper he raised three open questions needed to complete characterization of considered notions:
\begin{enumerate}[(1)]
\item\label{q:1} Does the existence of an uncountable $\omega$--scrambled set imply distributional chaos?
\item\label{q:2} Does the existence of an uncountable $\omega$--scrambled set imply existence of an infinite LY--scrambled set? 
\item\label{q:3} Does distributional chaos imply the existence of an infinite LY--scrambled set?
\end{enumerate}

Recently, in \cite{Drw18} Drwi\k{e}ga answered questions \eqref{q:1} and \eqref{q:3}. Therefore question \eqref{q:2} is the only which still remains open.
In the present paper we use some special properties of Sturmian subshift to construct appropriate example 
and prove the following Theorem.

\begin{thm}\label{thm:main}
There exists a continuous self-map $f$ on dendrite such that:
	\begin{enumerate}[(1)]
		\item  $f$ has an uncountable $\omega$-scrambled set,
		\item $f$ does not have an infinite LY--scrambled set.
	\end{enumerate}
\end{thm}

This provides a negative answer to this question~\eqref{q:2}.

\section{Definitions and notations} \label{sec:def}
Throughout
this paper $\N$ denotes the set $\{1,2,3, \dots \} $ and $\N_0 = \N \cup \{0\}.$ By a \emph{dynamical system} we mean a pair $(X,f)$, where $(X,\rho)$ is a compact metric space and $f$ is a continuous map from $X$ to itself.  
The \emph{orbit} of $x \in X$ is the set $O_f(x):=\{f^{k}(x): k\geq 0\}$, where $f^k$ stands for the $k$-fold composition of $f$ with itself. 
For $x \in X$, its $\omega$-$limit$ $set$ is defined by $\omega_f(x):=\cap_n \overline{O_f(f^n(x))}.$
A set $A \subset X$ is \emph{invariant} under $f$ if $f(A) \subset A$ and \emph{minimal} for $f$ if $A$ is nonempty, closed, invariant under $f$, and does not contain any proper subset which satisfies these three conditions. A dynamical system $(X,f)$ is \textit{minimal} if $X$ is a minimal set for $f$. It is known that $(X,f)$ is minimal if and only if every $x\in X$ has dense orbit or, equivalently, $\omega_f (x) =X$ for each $x \in X$. 
By $S^1$ we denote the unit circle identified with the interval $[0,1)$. For $\alpha \in  S^1$ the rotation by $\alpha$ is denoted by$R_\alpha$, that is $R_{\alpha}(x)={\text {mod}}_1 (x+\alpha)$ for $x \in S^1$. It is known that if $\alpha$ is irrational, then $(S^1,R_\alpha)$ is minimal (e.g. see \cite{Kur03}).

Let $X$ and $Y$ be compact metric spaces and let $f\colon X \to X$ and $g\colon Y \to Y$ be continuous map. If there is a continuous surjective map $\phi\colon X \to Y$ with $\phi \circ f= g \circ \phi,$ then $f$ and $g$ are \emph{semiconjugate} (by $\phi$). The map $\phi$ is called a \emph{semiconjugacy} or a \emph{factor map}, the map $g$ is called a \emph{factor} of $f$ and the map $f$ is called an \emph{extension} of $g$. If $\phi$ in the definition above is a homeomorphism then we call it a \emph{conjugacy}.

A pair of two different points $(x,y)\in X^2$ is
\begin{itemize}
 \item  \emph{proximal} if $$\liminf_{n \to \infty} \rho(f^n(x), f^n(y))=0,$$ 
 \item \emph{asymptotic} if $$\lim_{n \to \infty} \rho(f^n(x), f^n(y))=0,$$ 
 \item \emph{scrambled} or \emph{Li-Yorke} if it is proximal but not asymptotic.
\end{itemize}

A set $S \subseteq X$ is \emph{LY--scrambled} for $f$, if it contains at least two distinct points and every pair of distinct points in $S$ is scrambled. We say that $f \colon X \to X$ is  $\text{LY}$ chaotic if there exists a LY--scrambled set (this definition derives from \cite{Liyor75}). 
We say after \cite{Li93},  that a set $S \subseteq X$ is called $\omega$-scrambled for $f$ if it contains at least two points and for any $x,y \in S$ with $x\neq y,$ we have
\begin{enumerate}[(i)]
	\item $\omega_f (x) \setminus \omega_f (y)$ is uncountable,
	\item $\omega_f (x) \cap \omega_f (y)$ is nonempty,
	\item $\omega_f (x) $ is not contained in the set of periodic points.
\end{enumerate}
The map $f$ is $\omega$-$chaotic$ if there is at least a two-point $\omega$-scrambled set for $f$.

An \emph{arc} is any topological space homeomorphic to the interval $[0,1]$. 
A \emph{continuum} is a nonempty connected compact metric space. A \emph{dendrite} is a locally connected continuum containing no subset homeomorphic to the circle.
Let $X$ be a compact arcwise connected metric space and $v\in X$.
We say that $S\subset X$ is an \emph{$n$-star with center $v\in S$}
if there is a continuous injection $\varphi\colon S\to \C$ such that $\varphi(v)=0$
and $\varphi(S)=\{r \exp(\frac{2 k \pi i}{n}): r\in[0,1],\ k=1,2,\dotsc,n\}$.
The \emph{valence} of $v$ in $X$, denoted by $val(v)$,
is the number (which may be $\infty$)
$$\sup\{n\in\N\colon\text{there exists an }n\text{-star with center }
v\text{ contained in }X\}.
$$
The point $v$ is called an \emph{endpoint} of $X$ if $val(x)=1$,
and a \emph{branching point} of $X$ if $val(x)\geq 3$.
\emph{The Gehman dendrite} is a dendrite $G$ having the set of end points homeomorphic to the Cantor ternary set in $[0,1]$ such that all branching points  $G$ have valence $3$.

Finally, let us present some standard notation related to symbolic dynamics. 
Let $\A$ be any finite set (an \emph{alphabet}) and let $\A^*$ denote the set of all finite \emph{words} over $\A$ including the empty word. 
For any word $w \in \A^*$ we denote by $|w|$ the length of $w$, that is the number of letters which form this word. 
If $w$ is the empty word then we put $|w|=0$. An \emph{infinite word} is a mapping $w: \N \to \A$, in other words it is an infinite sequence $w_1 w_2 w_3 \dots$ where $w_i \in \A$ for any $i \in \N$. The set of all infinite words over an alphabet $\A$ is denoted by ${\A}^{\N}$. 
We endow  ${\A}^{\N}$ with the product topology of discrete topology on $\A.$ 
By $0^{\infty}$ we will denote the infinite word $0^{\infty}=000 \dots \in {\A}^{\N}.$ 
If $x \in \A^\N$ and $i,j \in \N$ with $i \leq j$ then we denote $x_{[i,j)}=x_{i}x_{i+1} \dots x_{j-1}$ (we agree with that $x_{[i,i)}$ is empty word)
and given $X \subset {\A}^{\N}$
by $\mathcal{L}(X)$ we denote \emph{the language} of $X$, that is, the set $\mathcal{L}(X):=\{x_{[1,k)}: x \in X, k> 0 \}$.  We write $\mathcal{L}_n(X)$ for the set of all words of length $n$ in $\mathcal{L}(X).$ 
If $u_k$ is a sequence of words such that $|u_k|\longrightarrow \infty$ then we write
$z=\lim_{k\to \infty} u_k$ if the limit $z=\lim_{k\to \infty} u_k 0^\infty$ exists in $\A^\N$.
Let $n \in \N$ and $\sigma$ a shift map defined on ${\A}^{\N}$ by $$(\sigma(x))_i=x_{i+1} \text{  for } i\in \N.$$ By  $\Sigma^{+}_{n}$ we denote a dynamical system formed by $ (\{0, \dots, n-1\}^{\mathbb{N}} , \sigma).$ 
If $S \subset {\A}^{\N}$ is nonempty, closed and $\sigma$-invariant then $S$ together with the restriction $\sigma |_{S} \colon S \to S$ (or even the set $S$) is called a \emph{subshift} of $\Sigma^{+}_{n}.$ Recall that the space ${\A}^{\N}$  is metrizable by the metric  $\rho: {\A}^{\N} \times {\A}^{\N} \to \R$ given for $x,y \in {\A}^{\N}$ by
$$ \rho(x,y)= \begin{cases} 2^{-k}, &\text{if } x \neq y, \\ 0, &\text{otherwise} \end{cases} $$ where $k$ is the length of maximal common prefix of $x$ and $y$, that is $k=\max \{ i \geq 1: x_{[1,i)}=y_{[1,i)}\}.$

\section{The Gehman Dendrite}\label{sec:lem}
Let us recall the construction of a continuous dendrite map from \cite{KoKor11}. Let $G$ be the Gehman dendrite and fix any point $p\in G$ with $val(p)=2$. 
For any distinct $a,b\in G$ we will denote by $[a,b]$ the unique arc in $G$ joining these points.
It is easy to see that branching points in $G$ can be arranged in such a way that each of the following intervals contains branching points (or point $p$) exactly at endpoints:
$B_0=[p,p_0], B_1=[p,p_1]$, and for every $n \in \N$, $B_{i_1 i_2  \dots i_n}=[p_{i_1 i_2 \dots i_n}, p_{i_1 i_2 \dots  i_{n+1}}]$ where every $i_k$ is either $0$ or $1$. 
It is well known that by the above construction we will cover all points of $G$ but endpoints (see Figure~\ref{fig:GehB}).
Furthermore, every point $x \in E$ can be uniquely associated to a sequence of zeros and ones $i_1 i_2 i_3 \dots$ in such way that it is the limit of the codes of the arcs converging to the point.
\begin{figure}[htp]
	\centering
		\includegraphics[width=0.7\linewidth]{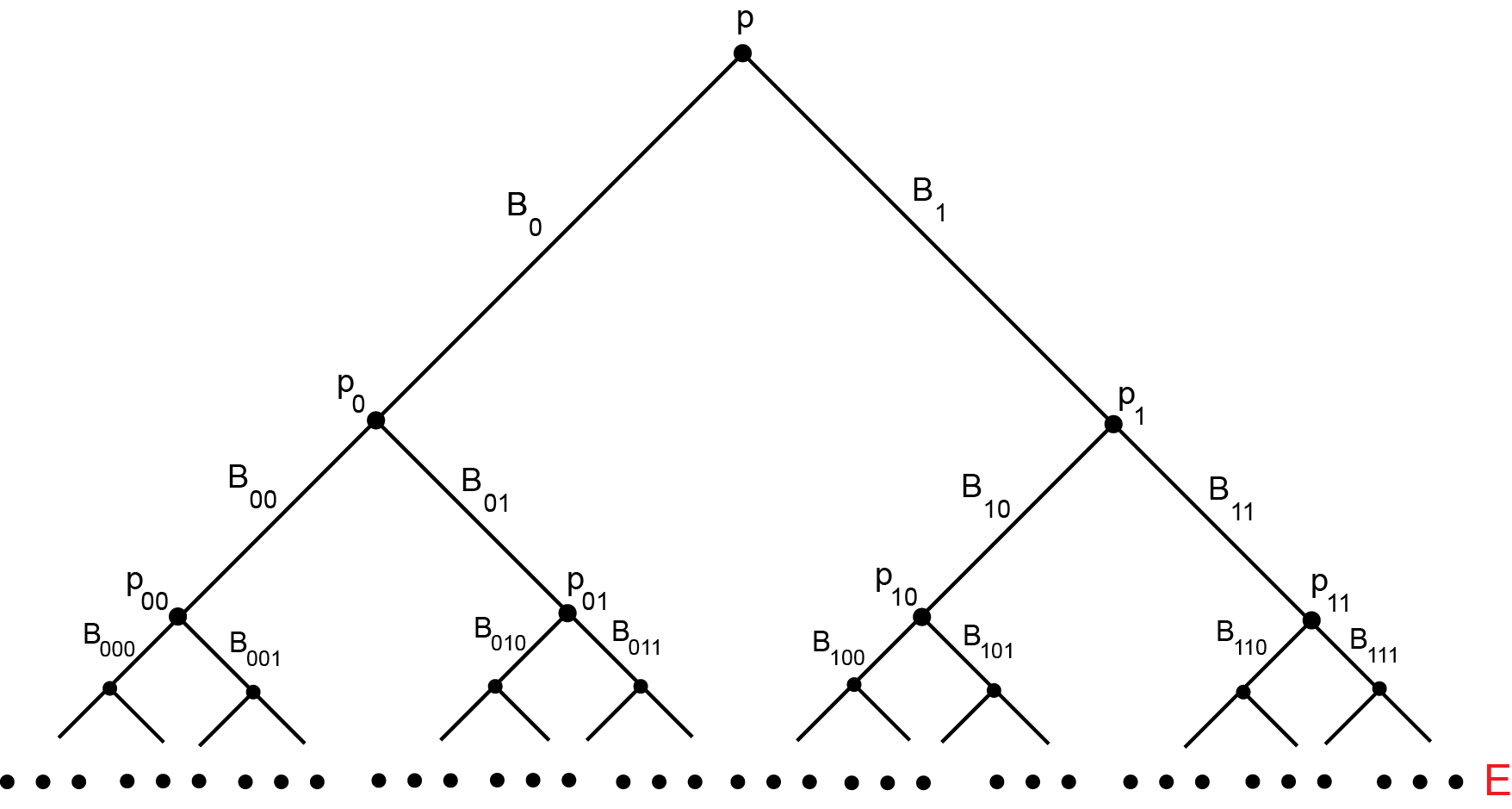}
	\caption{The Gehman dendrite}
	\label{fig:GehB}
\end{figure} 
 
 We define a continuous map $g$ on a dendrite $G$ in the following way. Let $g(B_0)=g(B_1)=\{p\}$. For every $i_1,i_2, \dots, i_n$, let $g |_{B_{i_1 i_2  \dots  i_n}}: B_{i_1 i_2  \dots  i_n} \to B_{i_2 i_3  \dots i_n}$ be a homeomorphism such that $g(p_{i_1 i_2  \dots  i_n})=p_{i_2 i_3  \dots  i_n}$, and let $g$ act on $E$ as the shift map on the space $\Sigma_{2}^+$.
Let $X$ be a closed $g$-invariant subset of $E$. Denote $$D_X= \bigcup_{x_{\xi} \in X} [x_{\xi},p]$$ and $$f=g |_{D_X}.$$ 
Now, let us recall usefull lemmas from \cite{Are01} about the Gehman dendrite.
\begin{lem}\label{lem:subdendrite}
If $X\subset \Sigma^{+}_{n}$ is a subshift then the set $D_X$ is an $f$-invariant subdendrite of the Gehman dendrite $G$.
\end{lem}

\begin{lem}\label{lem:Geh}
If $X$ is a closed and nonempty subset of  $\Sigma_{2}^+$ without isolated points, then the set $D_X= \bigcup_{x_{\xi} \in X} [x_{\xi},p] \subset G$ is homeomorphic with the Gehman dendrite. 
\end{lem}

\section{Sturmian subshift and the main example}\label{sec:constr}
A Sturmian subshift is an extension of an irrational rotation $(S^1, R_\alpha)$ generated by a particular interval cover (see \cite{Kur03}). Every Sturmian subshift is minimal and has zero topological entropy.

Let $\alpha \in (0,1)$ be irrational number and fix cover $[0,1/4)$, $[1/4,1)$ of $S^1$. 
Denote by $S_\alpha = \overline{ O(A(\alpha)) }$ the  Sturmian subshift
defined by the \emph{itinerary} $A(\alpha)$ of a point $\alpha$, that is the infinite sequence of symbols such that
$$A(\alpha)_i=0 \iff R^i_\alpha (\alpha) \in [0,1/4)$$ and
$$A(\alpha)_i=1 \iff R^i_\alpha (\alpha) \in [1/4, 1)$$ for $i \in \N.$

\begin{lem}\label{lem:Sclosed}
Let $C \subset \R \setminus \Q$ be a Cantor set.
Then the set $$S= \bigcup_{\alpha \in C} S_{\alpha}$$ is closed.
Furthermore, if $x_n \in S_{\alpha_n}$, $\lim_{n\to \infty} \alpha_n =\alpha\in C$ and $\lim_{n\to \infty}x_n=x$ exists, then $x\in S_\alpha$.
\end{lem}

\begin{proof}
Fix any sequence $(x_n) \subset S$ such that $\lim_{n \to \infty} x_n = x$ exists. There exists a sequence $(\alpha_n) \in C$ such that $x_n \in S_{\alpha_n}$ 
and going to a subsequence we may also assume that the following limit exists $\lim_{n\to \infty} \alpha_n =\alpha\in C$.

Let $\pi_n\colon S_{\alpha_n}\to S^1$ be the standard projection (almost 1-1 conjugacy) and let $\pi \colon S_\alpha \to S^1$.
Denote $z_n=\pi_n(x_n)$ and going to a subsequence assume that $z=\lim_{n\to \infty}z_n$ exists. We may also assume that all $z_n$ belong to a small one sided neighborhood of $z$
and converge monotonically with respect to that neighborhood. Let us consider two cases.

If $R_\alpha^i(z)\not\in \{0,1/4\}$ for every $i\in \N_0$ then there exists a unique point $y$ such that $\pi(y)=z$. Furthermore, for every $N$ there exists $k$
such that if $i\leq k$ and $n\geq N$ then $R_{\alpha_n}^i(z_n)\not\in \{0,1/4\}$. Clearly, for $n$ sufficiently large (or in other words, $z_n$ sufficiently close to $z$) we have $x_n(i)=y(i)$. This immediately implies that $y=x$ and so the claim holds.

For the second case, assume that $R_\alpha^i(z)\in \{0,1/4\}$ for some $i$. Since $\alpha$ is irrational number, there is at most one such integer $i$. For simplicity assume that $i=0$.
Then, by arguments as above we see that for every $i\neq 0$ there is $N$ such that $x_n(i)=y(i)$ for every $n\geq N$. This shows that $x(i)=y(i)$ for every $i\neq 0$.
But since $z\in \{0,1/4\}$, there exists $\hat{y}\in S_\alpha$ such that $\hat{y}(i)=y(i)$ for $i\neq 0$ and $y(0)=1-\hat{y}(0)$. In particular, either $x=y$ or $x=\hat{y}$.
In any case $x\in S_\alpha$, completing the proof.
\end{proof}

\begin{rem}
Sturmian system $S_\alpha$ does not have Li-Yorke pairs. Simply, if $\pi \colon S_\alpha \to S^1$
is factor map and $x,y\in S_\alpha$ is proximal then $\pi(x)=\pi(y)$. But in that case the pair is asymptotic.
\end{rem}

Let us recall Lemma 3.1 from \cite{Pik07} which is a slight generalization of Lemma 2.2 in \cite{Li93}:

\begin{lem}\label{lem:Pikula}
Let $a=a_1a_2a_3 \dots , b=b_1b_2b_3 \dots \in \{0,1\}^{\N}$. Define the following operation:
$$a \diamond b := a_1b_1a_1a_2b_1b_2a_1a_2a_3b_1b_2b_3 \dots$$
Then
\begin{enumerate}[(1)]
\item \begin{align*} \omega (a \diamond b) \supseteq \overline{O(a)} \cup \overline{O(b)}, \end{align*}
\item \begin{align*}
\omega (a \diamond b) \subseteq \overline{O(a)} \cup \overline{O(b)} 
& \cup \{a_i \dots a_j b_1b_2 \dots \colon j \geq i \geq 1\}  \\
& \cup \{b_i \dots b_j a_1a_2 \dots \colon j \geq i \geq 1\}.
\end{align*}
\end{enumerate}
\end{lem}

Now we are ready to start our construction of a subshift with a special structure.
Fix any Cantor set $A\subset [0,1/2)\setminus \Q$ and $\beta \in [0,1/2)\setminus (\Q\cup A)$.
Denote $S:=\bigcup_{\alpha\in A}S_\alpha$ and $Z=S_\beta$. By Lemma~\ref{lem:Sclosed} the set $S$ is a subshift and it is also clear that $S\cap Z=\emptyset$.
Since map $\psi \colon (Z,\sigma)\to (S^1, R_\beta)$ is $1-1$ on a residual set and every uncountable Borel set contains a Cantor set (e.g. see \cite{Sri98}), there exists a Cantor set $\hat A\subset Z$.
But any two Cantor sets are homeomorphic, we can index elements of $\hat A$ in the following way $\hat A= \{b_\alpha:\alpha\in A\}\approx A$.
For every $\alpha\in A$ let $\pi_\alpha\colon S_\alpha \to S^1$ be the standard factor map, and let $a_\alpha=\pi_\alpha^{-1}(1/8)$. Note that $a_\alpha$
is uniquely defined, since $R^i_\alpha(1/8)\not\in \{0,1/4\}$ for every $i\in \N_0$. Finally put 
$$
x_{\alpha}=a_\alpha \diamond b_\alpha.
$$
By Lemma~\ref{lem:Pikula} we see that
$$\omega(x_{\alpha})=\{S_\alpha, Z\}  \cup \{a_i \dots a_j b_1b_2 \dots \colon j \geq i \geq 1\} \cup \{b_i \dots b_j a_1a_2 \dots \colon j \geq i \geq 1\}.$$
where $a_\alpha=a_1a_2\ldots$ and $b_\alpha=b_1b_2\ldots$.

\begin{lem}\label{lem:scr-triple}
For every $\alpha\in A$ the set $\overline{O(x_\alpha)}$ does not contain scrambled set with more than two points.
\end{lem}
\begin{proof}
It is clear, that any two points $x,y\in O(x_\alpha)$ are not proximal, because $x_\alpha$ is not eventually periodic.
Similarly, if $x\in S_\alpha$ and $y\in Z$, they cannot be proximal. This implies that the only possibility when $x,y$ are Li-Yorke pair,
is that $x\in O(x_\alpha)$ and $y\in \omega(x_\alpha)$. Indeed, there is no three points scrambled set.
\end{proof}

Denote
\begin{equation}
\mathbb{X}=\overline{\bigcup_{\alpha\in A} O(x_{\alpha})}\label{eq:X}
\end{equation}

\begin{lem}\label{lem:chaos}
The shift $\mathbb{X}$ has an uncountable $\omega$-scrambled set and does not have an infinite scrambled set.
\end{lem}
\begin{proof}
First we claim that
$$
\mathbb{X}=\overline{\bigcup_{x_{\alpha}} O(x_{\alpha})}= \bigcup_{x_{\alpha}} O(x_{\alpha}) \cup \omega(x_{\alpha})=\bigcup_{x_{\alpha}} \overline{O(x_{\alpha})}.
$$
By the argument in the proof of Lemma~\ref{lem:Sclosed} we see that if $x_{\alpha_n}\to x$ and $\alpha=\lim_{n\to \infty}\alpha_n$ then $\lim_{n\to \infty}x_{\alpha_n}=x_\alpha$,
mainly because all points $x_{\alpha_n}$ project onto point $1/8\in S^1$ which has singleton fiber (we can repeat argument from the proof of Lemma~\ref{lem:Sclosed}). 

Now, fix any sequence of points $(x_n)\subset \bigcup_{\alpha\in A} O(x_{\alpha})$ and let $x=\lim_{n\to \infty} x_n$.
We~assume, going to a subsequence if necessary that $x_n\in O(x_{\alpha_n})$ and $\lim_{n\to \infty}\alpha_n=\alpha$.

There are four possible cases (after going to a subsequence):
\begin{enumerate}[(1)]
	\item For every $n$  there exists $a_n \in S_\alpha$ such that $d(a_n,x_n)<1/n$. In this case, we may also assume that $\lim_{n\to \infty}a_n=a$ exists.
	But then $a\in S_\alpha$ by the arguments in Lemma~\ref{lem:Sclosed} and we also must have $a=x$, so the claim holds.
	\item For every $n$  there exists $b_n \in Z$ such that $d(b_n, x_n)<1/n$, and then $x\in Z$.
	\item For every $n$ there exist  $i_n<j_n$ such that $x_n=a^n_{i_n} \dots a^n_{j_n} b^n_1 b^n_2$ where $b^n=b_{\alpha_n}\in Z$ and $a^n=a_{\alpha_n}\in S_{\alpha_n}$ are points from the definition. First suppose that $j_n-i_n\to \infty$. Then $\lim_{n\to \infty}x_n=\lim_{n\to \infty}\sigma^{i_n}(a_{\alpha_n})\in S_\alpha$ by Lemma~\ref{lem:Sclosed}. In the second case, when $i_n-j_n\leq k$ for every $n$, there exists a word $w$ (which is a subword of every $x_{n}$, so also a subword of $x_\alpha$), such that $x_{n}=w b^n_1b^n_2 \ldots$ and then $\lim_{n\to \infty} x_{n}=w b_1b_2\ldots$ where $b_\alpha=b_1b_2\ldots$ so also in this case the claim holds.
	\item For every $n$ there exists  $i_n>j_n$ such that $x_n=b^n_{i_n} \dots b^n_{j_n} a^n_1 a^n_2$. This case is analogous to the previous one.
\end{enumerate}
The claim is proved. Now the rest of proof is simple. Namely, if $D$ is a scrambled set with at least three points, then by Lemma~\ref{lem:scr-triple}
we see that there are distinct $\alpha,\gamma\in A$ such that $x\in \overline{O(x_\alpha)}$ and 
$y\in \overline{O(x_\gamma)}$ form a Li-Yorke pair (that is $x\not\in \overline{O(x_\gamma)}$ and $y\not\in\overline{O(x_\alpha)}$).
Since $\dist(S_\alpha ,S_\beta)>0$, the only possibility is that points $(x,y)$ are proximal through set $Z$, that is, there is $z\in Z$ and $j_n$ such that $d(\sigma^{j_n}(x),z)<1/n$ and  $d(\sigma^{j_n}(y),z)<1/n$. But then the only possibility by Lemma~\ref{lem:Pikula} is that $\sigma^{j_n}(x)=\sigma^{j_n+p}(x_\alpha)$ and $\sigma^{j_n}(y)=\sigma^{j_n+q}(x_\gamma)$, for some numbers $p, q\in \N_0$. But it is impossible, since when $x_\alpha$ is close to $Z$ it follows the trajectory of $b_\alpha$ while $x_\gamma$ follows the trajectory of $b_\gamma$. But $b_\alpha$ and $b_\gamma$ belong to distinct trajectories from singleton fibers through projection of $Z$ onto $S^1$, therefore their itineraries $\sigma^p(b_\alpha)$, $\sigma^q(b_\gamma)$ are distal, in particular cannot approach common point $z$ simultaneously.
Indeed, there is not scrambled set with more than two points.

On the other hand, the set $C=\{x_\alpha : \alpha \in A\}$ is $\omega$-scrambled, because $Z\subset \bigcap_{\alpha \in A}\omega(x_\alpha)$
and for $\alpha\neq \gamma$ we have $S_\alpha \subset \omega(x_\alpha)\setminus \omega(x_\gamma)$.
The proof is completed.
\end{proof}

\begin{proof}[Proof of Theorem~\ref{thm:main}]
Let $\mathbb{X}$ be provided by \eqref{eq:X}. If we put $D=D_{\mathbb{X}}$ then by Lemma~\ref{lem:subdendrite} $D$ is a Gehman dendrite and the induced map $f\colon D \to D$ is well defined continuous surjection.  We may view $\mathbb{X} \subset D$ as a set of endpoints of $D$ invariant for $f$.
Furthermore, there is a unique fixed point $p\in D\setminus \mathbb{X}$ such that, if $x\not\in \mathbb{X}$ then $\lim_{n\to \infty}f^n(x)=p$.
Then it is clear that the only Li-Yorke pairs for $f$ are these contained in $\mathbb{X}\subset D$. The statement of Lemma~\ref{lem:chaos} completes the proof.
\end{proof}

\section*{Acknowledgements}

Research of P. Oprocha was supported by National Science Centre, Poland (NCN), grant no. 2015/17/B/ST1/01259.


\begin{thebibliography}{10}
	\bibitem{Are01} D. Ar\'evalo, W. Charatonik, P. Pellicer Covarrubias, L. Sim\'on, \textit{Dendrites with a closed set of end points}. Topology Appl. \textbf{115} (2001), no. 1, 1--17.
	\bibitem{Drw18}T.~Drwi\k{e}ga, \emph{Dendrites and chaos}, to appear in Internat. J. Bifur. Chaos Appl. Sci. Engrg., to appear
	\bibitem{Koc12} Z.~Ko\v{c}an, \emph{Chaos on one-dimensional compact metric spaces}, Internat. J. Bifur. Chaos Appl. Sci. Engrg. \textbf{22}, article id: 1250259 (2012).
	\bibitem{KoKor11} Z.~Ko\v{c}an, V.~Korneck\'{a}-Kurkov\'{a} \&  M.~ M\'{a}lek \emph{Entropy, horseshoes and homoclinic trajectories on trees, graphs and dendrites}, Ergod. Th. Dyn. Sys. \textbf{31}, (2011) 165--175, Erratum: 177--177.
	\bibitem{Kur03}P.~K\r{u}rka \emph{Topological and Symbolic Dynamics}, Societe Mathematique de France, 2003.
	\bibitem{Li93}S. H.~Li, \emph{{$\omega$}-chaos and topological entropy}, Trans. Amer. Math. Soc. \textbf{339} (1993), 243--249.
	\bibitem{Liyor75}T. Y.~Li, J.~Yorke, \emph{Period three implies chaos}, Amer. Math. Monthly \textbf{82} (1975), 985--992.
	\bibitem{Pik07} R.~Piku\l{}a, \emph{On some notions of chaos in dimension zero}, Colloq. Math. \textbf{107}  (2007), 167--177.	
	\bibitem{Sri98} S. M.~Srivastava, \emph{A Course on Borel Sets}, Graduate Texts in Mathematics, \textbf{vol. 180}, Springer-Verlag, New York, 1998.
\end{thebibliography}
\end{document}